\documentclass{article}

\usepackage[latin1]{inputenc}  
\usepackage[T1]{fontenc}
\usepackage{amsmath,amssymb,dcolumn,amsthm}
\usepackage{graphicx,color,transparent}
\usepackage{algorithm}
\usepackage{algorithmic}
\usepackage{float}
\makeatletter
\renewcommand{\maketag@@@}[1]{\hbox{\m@th\normalsize\normalfont#1}}%
\makeatother

\usepackage{tikz}
\usetikzlibrary{arrows}
\tikzstyle{block} = [draw, draw=black, line width = 1pt, rectangle,
    rounded corners=4pt,
    minimum height=3em, text width=.7\columnwidth,
    inner sep = 10pt]

\usepackage{ragged2e}

\usepackage{ankelhed,axearticle,axecontrol,axemath,axeopt,axeabrv,axeprog,axeIEEEarticle,ian_misc,ian_math}

\usepackage{ragged2e}

  \newcommand{\MPC}[1]{\mathcal{P}(#1)}
  \newcommand{\xh}[1]{\hat x_{#1}}
  \newcommand{\uh}[1]{\hat u_{#1}}
  \newcommand{\nx}{n_x}
  \newcommand{\nuu}{n_u}
  \newcommand{\p}{p}
  \newcommand{\Ph}[1]{\hat P_{#1}}
  \newcommand{\Kh}[1]{\hat K_{#1}}

  \newcommand{\Qxh}[1]{\hat Q_{x,#1}}
  
  \newcommand{\Quh}[1]{\hat Q_{u,#1}}
  \newcommand{\Ah}[1]{\hat A_{#1}}
  \newcommand{\Bh}[1]{\hat B_{#1}}
  \newcommand{\Qh}[1]{\hat Q_{#1}}
  \newcommand{\Vh}[1]{\hat V_{#1}}
    \newcommand{\MPCsub}[2]{\mathcal{P}_{#1}^{#2}(N_{#1}^{#2})}
  \newcommand{\uu}{u}
  \newcommand{\ub}[1]{\bar u_{#1}}
  \newcommand{\thh}{\bar t}
  \newcommand{\Vb}[1]{\bar V_{#1}}
  \newcommand{\Sc}{\mathsf{S}}
  \newcommand{\Fb}[1]{\bar F_{#1}}
  \newcommand{\Gb}[1]{\mathsf{\bar G_{#1}}}
  \newcommand{\Hb}[1]{\mathsf{\bar H_{#1}}}
  \newcommand{\Kb}[1]{\mathsf{\bar K_{#1}}}
  \newcommand{\Vk}[1]{C_{#1}}
  \newcommand{\Qb}[1]{\mathsf{\bar Q_{#1}}}

  \newcommand{\Mm}[1]{M_{#1}}
  \newcommand{\nr}{n_{\hat u}}
  \newcommand{\Gh}[1]{\hat G_{#1}}
  \newcommand{\Fh}[1]{\hat F_{#1}}
  \newcommand{\Hh}[1]{\hat H_{#1}}
  
\title{\LARGE \bf A Parallel Riccati Factorization Algorithm\\ with Applications to Model Predictive Control}

\author{Isak Nielsen and Daniel Axehill%
  \thanks{I. Nielsen and D. Axehill are with the Division of Automatic
    Control, Linköping University, SE-58183 Linköping, Sweden, \texttt{isak.nielsen@liu.se, daniel@isy.liu.se}.}}
    
\begin{document}


\begin{center}
\huge {A Parallel Riccati Factorization Algorithm\\ with Applications to Model Predictive Control} \\ \vspace{2mm}
\large Isak Nielsen, Daniel Axehill \\ \normalsize (Division of Automatic Control, Link\"oping University, Sweden (e-mail: \{isak.nielsen@liu.se, daniel@isy.liu.se\})
\end{center}

\textbf{Abstract}
Model Predictive Control (MPC) is increasing in popularity in industry as more efficient algorithms for solving the related optimization problem are developed. The main computational bottle-neck in on-line MPC is often the computation of the search step direction, \ie the Newton step, which is often done using generic sparsity exploiting algorithms or Riccati recursions. However, as parallel hardware is becoming increasingly popular the demand for efficient parallel algorithms for solving the Newton step is increasing. In this paper a tailored, non-iterative parallel algorithm for computing the Riccati factorization is presented. The algorithm exploits the special structure in the MPC problem, and when sufficiently many processing units are available, the complexity of the algorithm scales logarithmically in the prediction horizon. Computing the Newton step is the main computational bottle-neck in many MPC algorithms and the algorithm can significantly reduce the computation cost for popular state-of-the-art MPC algorithms.

\vspace{3mm}

\textbf{Keywords}
Model Predictive Control, Parallel Computation, Optimization, Riccati factorization

\section{Introduction}
\label{sec:intro}
One of the most widely used control strategies in industry today is Model Predictive Control (MPC). Some important reasons for its success
include that it can handle multi-variable systems and constraints on
control signals and state variables in a structured way~\cite{maciejowski2002predictive}. Each
sample of the MPC control loop consists of solving an optimization problem on-line, which requires efficient optimization algorithms.
However, similar linear algebra is also useful
off-line in explicit MPC solvers, where the optimal feedback is pre-computed. Depending on the type of system
and problem formulation, the optimization
problem can be of different types, where the most common variants are
linear MPC, nonlinear MPC and hybrid MPC. In most cases, the effort
spent in the optimization problems boils down to solving
Newton-system-like equations, which has led to that much focus in research has been spent on solving this type of system of equations efficiently when
it has the special form from MPC, see \eg\cite{jonson83:thesis,rao98:_applic_inter_point_method_model_predic_contr,hansson00:_primal_dual_inter_point_method,bartlett02:_quadr,vandenberghe02:_robus_full,akerblad04:_effic,axehill06:_mixed_integ_dual_quadr_progr_tail_mpc,axevanhan07:_relax_applic_mipc_compa_and_eff_compu,axehill08:thesis,axehill08:_dual_gradien_projec_quadr_progr,diehl09:_nonlin_model_predic_contr,nielsen13low_rank_updates}. 

In recent years, the demand for efficient parallel algorithms for solving the MPC problem has increased, and much effort in research has been spent on this topic~\cite{constantinides2009tutorial}. In \cite{soudbakhsh2013parallelized} an extended Parallel Cyclic Reduction algorithm is used to reduce the computation to smaller systems of equations that are solved in parallel. The computational complexity of this algorithm is reported to be $\Ordo{\log N}$, where $N$ is the prediction horizon. In \cite{laird2011parallel}, \cite{ZhuParallelNP} and \cite{reuterswardLic} a time-splitting approach to split the prediction horizon into blocks is adopted. The subproblems in the blocks are connected through common variables and are solved in parallel using Schur complements. The common variables are computed via a consensus step where a dense system of equations involving all common variables has to be solved sequentially.
In~\cite{o2012splitting} a splitting method based on Alternating Direction Method of Multipliers (ADMM) is used, where some steps of the algorithm can be computed in parallel. In \cite{stathopoulos2013hierarchical} an iterative three-set splitting QP solver is developed. In this method the prediction horizon is split into smaller sub problems that are in turn split into three simpler problems. All these can be computed in parallel and a consensus step using ADMM is performed to achieve the final solution. In~\cite{nielsen14_parallel_mpc_arxive} the first tailored algorithm for solving the Newton step in parallel for MPC  is presented. In that work several subproblems are solved parametrically in parallel by introducing terminal constraints on the final state in each subproblem. However, the structure in the subproblems are not exploited when the subproblems are solved.

The main contribution in this paper is the introduction of theory and algorithms for solving the Riccati recursion in parallel. The new algorithms are tailored for MPC problems and fully exploit the special structure of the KKT system for such problems. The classical serial Riccati method exploits the causality of the problem and for that reason it is not obvious that it can be split and parallelized in time, especially without involving some form of iterative consensus step. In this paper, it is shown that it in fact can be performed, and how it can be performed.
The main idea is to exploit the problem structure in time and divide the original MPC problem in smaller subproblems along the prediction horizon. The subproblems are condensed in parallel using Riccati recursions to create a new MPC problem of smaller size, \ie, with shorter prediction horizon and fewer control signals. This new MPC problem is solved, and the information that is needed to solve the subproblems independently is computed. Finally, all subproblems are solved independently in parallel. Hence, the Riccati recursion for the original problem has been performed in parallel.

In this article, $\posdefmats^n$ ($\possemidefmats^n$) denotes
symmetric positive (semi) definite matrices with $n$ columns. Furthermore, let $\intnums$
be the set of integers, and $\intset{i}{j} = \braces{i,i+1,\hdots,j}$. Symbols in sans-serif font (\eg $\timestack{x}$) denote vectors or matrices of stacked element, $I$ denotes the identity matrix of appropriate dimension, and the product operator is defined as
\begin{equation}
\prod_{t=t_1}^{t_2} A_t = \begin{cases} 
A_{t_2} \cdots A_{t_1}, \; t_1 \leq t_2 \\
I, \; t_2 > t_1.
\end{cases}
\end{equation}

The paper is organized as follows. In Section~\ref{sec:prob_form} the problem description is formulated and Section~\ref{sec:ric_rec} presents the algorithms for solving this problem using the serial Riccati recursion. In Section~\ref{sec:prob_red} the original problem is split into smaller subproblems, and it is also shown how to reduce these into a smaller MPC problem. Section~\ref{sec:par_ric} presents the parallel Riccati recursion, the algorithms and the numerical results for the implemented algorithms. Finally, Section~\ref{sec:conclusion} concludes the paper.



\section{Problem Formulation}
\label{sec:prob_form}
In this work linear MPC problems are considered, where the optimization problem that is solved in each sample is a convex quadratic program (QP) problem in the form

\begin{equation}
  \label{eq:min_problem}
 \minimize{
    &\frac{1}{2}\sum^{N-1}_{t=0}\begin{bmatrix}
    x_t^T & u_t^T
    \end{bmatrix} Q_t\begin{bmatrix}
    x_t \\  u_t
\end{bmatrix} 
     +\frac{1}{2}x^T_NQ_{x,N}x_N}
  {\timestack{x},\timestack{u}}
  {&x_0 = \bar x_0 \\
    &x_{t+1} = A_tx_t + B_t u_t, \; t \in \intset{0}{N-1} \\
    & u_t \in \mathcal{U}_t, \; t \in \intset{0}{N-1} \\
    &x_t \in \mathcal{X}_t, \; t \in \intset{0}{N}.}
\end{equation}
The equality constraints represent the dynamics equations of the system and $\mathcal{U}_t$ and $\mathcal{X}_t$ are the sets of feasible control signals and states, respectively. Let the following assumptions hold for all $t$
\begin{assumption}
\begin{equation}
Q_t = \begin{bmatrix}
Q_{x,t} & Q_{xu,t} \\ Q_{xu,t}^T & Q_{u,t}
\end{bmatrix} \in \possemidefmats^{n_x+n_u}, \; Q_{u,t} \in \posdefmats^{n_u}, \; Q_{x,N} \in \possemidefmats^{n_x}
\end{equation}
\end{assumption}
\begin{assumption}
 $\mathcal{X}_t = \mathbb{R}^{n_x}$ and $\mathcal{U}_t$ consists of constraints of the form $ 0 \leq  u_t$, \ie lower bounds on the control signal. \label{ass:X_U}
\end{assumption}
\begin{remark}
The theory presented in this paper can be used to solve more general MPC problems with linear penalty terms, affine dynamics and/or more general constraints on the control signals and constraints. The problem formulation~\eqref{eq:min_problem} and the constraints in Assumption~\ref{ass:X_U} have been chosen for notational brevity.
\end{remark}

There exists different methods for solving an MPC problem on the form~\eqref{eq:min_problem}, see \eg~\cite{nocedal06:num_opt}, where two common methods are interior-point (IP) methods and active-set (AS) methods. In IP methods, the inequality constraint functions are approximated with barrier functions, whereas the AS methods iteratively changes the set of inequality constraints that hold with equality until the optimal active set has been found. The main computational effort in both types is spent while solving Newton-system-like equations that often corresponds to an equality constrained MPC problem with prediction horizon $N$ (or to a problem with similar structure). Note that this problem is also an important part of non-linear MPC algorithms as well as hybrid MPC algorithms. In this paper the equality constrained MPC problem will be denoted $\MPC{N}$, and has the structure
\begin{equation}
 \label{eq:org_eqc_problem} \minimize{
    &\frac{1}{2}\sum^{N-1}_{t=0}\begin{bmatrix}
    x_t^T & \uu_t^T
    \end{bmatrix} Q_t\begin{bmatrix}
    x_t \\ \uu_t
\end{bmatrix} 
     +\frac{1}{2}x^T_NQ_{x,N}x_N}
  {\timestack{x},\timestack{u}}
  {&x_0 = \bar x_0 \\
    &x_{t+1} = A_tx_t + B_t\uu_t, \; t \in \intset{0}{N-1}.}
\end{equation}
This problem is obtained from~\eqref{eq:min_problem} by fixing some of the inequality constraints as in an AS method, and disregarding the rest of the inequality constraints. The control signals that are fixed to zero (the corresponding inequality constraints are fixed) are removed from the problem. 
\begin{remark}
The matrices in~\eqref{eq:org_eqc_problem} might have different dimensions than in~\eqref{eq:min_problem}. For the remaining part of the paper, problems of the form~\eqref{eq:org_eqc_problem} are considered.
\end{remark}


\section{Standard Riccati Recursion}
\label{sec:ric_rec}
The solution to the equality constrained MPC problem~\eqref{eq:org_eqc_problem} is computed by solving the set of linear equations given by the associated KKT system. For this problem structure, the KKT system has a very special form that is almost block diagonal and can be factored efficiently using a Riccati factorization that can be computed using Riccati recursions. The Riccati factorization is used to factor the KKT coefficient matrix, followed by forward recursions to compute the primal and dual variables. Using Riccati recursions to solve the KKT system reduces the computational complexity from roughly $\Ordo{N^2}-\Ordo{N^3}$ to $\Ordo{N}$. For more background information on Riccati factorizations, see, e.g.,~\cite{jonson83:thesis},~\cite{rao98:_applic_inter_point_method_model_predic_contr} or~\cite{axehill08:thesis}.

Let the matrices $F_t,P_t,G_t,Q_{x,t}\in\Sm{n_x}_{+}$, $Q_{u,t}\in\Sm{n_u}_{++}$, ${H_t,Q_{x\uu,t}\in\Rm{n_x\times n_u}}$ and ${K_t\in\Rm{n_u \times
  n_x}}$. The Riccati factorization is then given by Algorithm~\ref{alg:factorization} and the forward recursions are given by Algorithm~\ref{alg:fwd_rec}-\ref{alg:fwd_rec_dual},~\cite{axehill08:thesis}. In Algorithm~\ref{alg:fwd_rec_dual} the dual variables corresponding to fixed inequality constraints are computed. Here $B_{v,t}$, $Q_{xv,t}$ and $Q_{\uu v,t}$ represent the parts of the respective matrices that correspond to the fixed control signals.
\begin{algorithm}[h!]
  \caption{Riccati factorization} \label{alg:factorization}
  \begin{algorithmic}[1]
    \STATE $P_N := Q_{x,N}$
    \FOR{$t=N-1,\ldots,0$}
    \STATE $F_{t+1} := Q_{x,t} + A^T_tP_{t+1}A_t$\label{alg:factorization:line:F} \\
    \STATE $G_{t+1} := Q_{\uu,t} + B^T_tP_{t+1}B_t$ \label{alg:factorization:line:G}\\
    \STATE $H_{t+1} := Q_{x\uu,t} + A^T_tP_{t+1}B_t$\label{alg:factorization:line:H} \\
    \STATE Compute and store a factorization of $G_{t+1}$.\label{alg:factorization:line:factorize_G}
    \STATE Compute a solution $K_{t+1}$ to \\ 
    $G_{t+1}K_{t+1} = -H^T_{t+1}$ \\
    \STATE $P_{t} := F_{t+1} - K^T_{t+1}G_{t+1}K_{t+1}$\label{alg:factorization:line:P}
    \ENDFOR
  \end{algorithmic}
\end{algorithm}

%
\begin{algorithm}[htbp!]
  \caption{Forward recursion}
  \label{alg:fwd_rec}
  \begin{algorithmic}[1]
  	\STATE $x_0 = \bar x_0$
    \FOR{$t = 0,\hdots,N-1$}
	\STATE $\uu_t = K_{t+1}x_t$
    \STATE $x_{t+1} = A_tx_t +
    B_t \uu_t$
    \STATE $\lambda_t = P_tx_t$
    \ENDFOR
    \STATE $\lambda_N = P_Nx_N$
  \end{algorithmic}
\end{algorithm}
\begin{algorithm}[htbp!]
  \caption{Forward recursion (Dual variables)}
  \label{alg:fwd_rec_dual}
  \begin{algorithmic}[1]
    \FOR{$t = 0,\hdots,N-1$}
    \STATE $\mu_t = Q_{xv,t}^Tx_t +
    B_{v,t}^T\lambda_{t+1} + Q^T_{\uu v,t}u_t$
    \ENDFOR
  \end{algorithmic}
\end{algorithm}


\section{Problem decomposition and reduction}
\label{sec:prob_red}
In Section~\ref{sec:ric_rec} serial algorithms for solving the MPC problem~\eqref{eq:org_eqc_problem} in $\Ordo{N}$ complexity using the Riccati recursion were presented. This section will introduce new theory to compute the Riccati recursion in parallel directly (non-iteratively) on several processing units with $\Ordo{\log N}$ complexity. To do this, the main idea is to divide the original problem into several smaller subproblems along the prediction horizon, see Fig.~\ref{fig:sub_prob_connection}. It will be shown that each subproblem $i$ can be solved independently of the others, provided that the initial value and $x_{0,i}=x_{\tilde N_{i-1}}$ and $P_{N_i,i} = P_{\tilde N_i}$ are known to the subproblem (for now, it is enough to realize that $P_{N_i,i}$ transfers information backwards in time from subproblem $i+1$ to $i$). To compute $x_{0,i}$ and $P_{N_i,i}$, the subproblems are individually condensed using the Riccati recursion, and combined into an MPC problem of smaller size (\ie shorter prediction horizon and lower control signal dimension). Solving this new smaller MPC problem with the Riccati recursion computes $P_{N_i,i}$ and $x_{0,i}$. When these are known, the subproblems can be solved independently in parallel.

The main focus of this section will be how to split the original problem~\eqref{eq:org_eqc_problem} in time into several smaller subproblems (Section~\ref{subsec:sub_prob}), how to condense the subproblems efficiently using the Riccati recursion (Section~\ref{subsec:red_sub_prob}) and how to form the reduced MPC problem (Section~\ref{subsec:red_mpc_prob}).



\subsection{Splitting the MPC problem into subproblems}
\label{subsec:sub_prob}
By examining the Riccati factorization given by Algorithm~\ref{alg:factorization} and the forward recursion given by Algorithm~\ref{alg:fwd_rec}, it is clear that $P_t$ transfers information backwards in time, and the state $x_{t}$ transfers information forward in time. Hence, the problem can be split in time into smaller batches that exchange information with the adjacent batches via $P_{t}$ and $x_{t}$ at the end points, see Fig.~\ref{fig:sub_prob_connection}. 
\begin{figure}
\centering
\def\svgwidth{0.95\columnwidth}
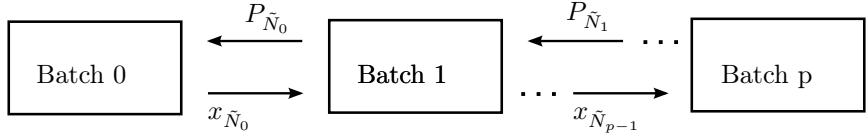
\caption{The MPC problem can be divided into smaller batches (in time), where the batches are exchanging information via $P_t$ and $x_t$ at the division points.}
\label{fig:sub_prob_connection}
\end{figure}

To decompose the problem, let the prediction horizon be split such that $\timestack{x}$ and $\timestack{u}$ are divided into $\p+1$ batches
\begin{align}
&\begin{bmatrix}
x_0^T&\cdots &x_{\tilde N_0}^T
\end{bmatrix}^T, \ldots, \begin{bmatrix}
x_{\tilde N_{\p-1}}^T &\cdots &x_{\tilde N_\p}^T
\end{bmatrix} ^T, \\
&\begin{bmatrix}
\uu_0^T &\cdots & u_{\tilde N_0-1}^T
\end{bmatrix}^T,  \ldots, \begin{bmatrix}
\uu_{\tilde N_{\p-1}}^T &\cdots &\uu_{\tilde N_\p-1}^T
\end{bmatrix} ^T.
\end{align}
Note that the last state $x_{\tilde N_i}$ in batch $i$ is the same as the first state in batch $i+1$. Now introduce the batch-wise variables
\begin{align}
\timestack{x_i} &\triangleq \begin{bmatrix}
x_{0,i}^T & \cdots & x_{N_i,i}^T
\end{bmatrix}^T = \begin{bmatrix}
x_{\tilde N_{i-1}}^T & \cdots & x_{\tilde N_i}^T
\end{bmatrix}^T \\
\timestack{u_i} &\triangleq \begin{bmatrix}
\uu_{0,i}^T & \cdots & \uu_{N_i-1,i}^T
\end{bmatrix}^T = \begin{bmatrix}
\uu_{\tilde N_{i-1}}^T & \cdots & \uu_{\tilde N_i-1}^T
\end{bmatrix}^T,
\end{align}
for $i \in \intset{0}{p}$, where $N_i$ is the length of batch $i$.
%
%
By inspection of Fig.~\ref{fig:sub_prob_connection}, the original problem can then be decomposed into $\p+1$ smaller MPC problems $i\in \intset{0}{p}$, with initial value $\xh{i}$ and terminal state cost $P_{N_i,i}$, on the form
\begin{equation}
\minimize{\frac{1}{2}&\sum_{t=0}^{N_i-1}\begin{bmatrix}
    x_{t,i}^T & \uu_{t,i}^T
    \end{bmatrix} Q_{t,i}\begin{bmatrix}
    x_{t,i} \\  \uu_{t,i}
\end{bmatrix}\\&+\frac{1}{2}x_{N_i,i}^TP_{N_i,i}x_{N_i,i}}{\timestack{x}_i,\timestack{u}_i}{x_{0,i} &= \xh{i} \\ x_{t+1,i}&=A_{t,i}x_{t,i}+B_{t,i}\uu_{t,i}, \; t\in \intset{0}{N_i-1}.} \label{eq:sub_problem_dscr}
\end{equation}
Note that for the final batch the terminal constraint is $P_{N_p,p}=Q_{x,N}$. 
%
Provided that the optimal value of $P_{N_i,i}$ (\ie $P_N$ in Algorithm~\ref{alg:factorization} for batch $i$) and $\xh{i}$ are known, these individual subproblems can be solved completely independently of each other using $\p+1$ Riccati recursions. 


\subsection{Reducing the size of a subproblem}
\label{subsec:red_sub_prob}
Even when $P_{N_i,i}$ is not known, it is possible to work on the subproblems individually to reduce their sizes. This can be done separately for the $\p+1$ sub problems, which opens up for a structure which can be solved in parallel. The core idea with this approach is that the unknown $P_{N_i,i}$ will indeed influence the solution of the subproblem, but the degree of freedom is often very limited compared to the dimension of the full control signal vector. It will be shown in this section that the structured perturbation from $P_{N_i,i}$ only introduces $\nr \leq \nx$ degrees of freedom, and hence that the subproblem can be reduced to depend only on the initial state $\xh{i}$ and the freedom in the structured perturbation, of dimension $\nx$ and $\nr$ respectively. Furthermore, it will be shown how the reduced subproblems can be combined into a new MPC problem $\MPC{\p}$ of smaller size, \ie with $p < N$ and lower control signal dimension. This is summarized in Theorem~\ref{thm:reduce_subprob}, and the proof of this theorem is partly based on Lemma~\ref{lem:F_G_H_cost} where an expression for the cost-to-go at time $\bar t$  as a function of $x_{\bar t}$ and $P_{\bar t}$ is presented. For notational brevity, the subindices $i$ in~\eqref{eq:sub_problem_dscr} are omitted in Lemma~\ref{lem:F_G_H_cost} and Theorem~\ref{thm:reduce_subprob} and their proofs.

\begin{lemma}
\label{lem:F_G_H_cost}
Assume that $\uu_{t}=K_{t+1}x_{t} + \ub{t}$ for $t \in \intset{\bar t}{N-1}$, where $\bar u_t \in \mathbb{R}^{\nuu}$ is an arbitrary vector. Then the cost-to-go at $\bar t$ for the problem~\eqref{eq:sub_problem_dscr} with $P_{N}=0$ is
\begin{equation}
\Vb{} \left(x_{\bar t},\timestack{\bar \uu}\right)\triangleq \frac{1}{2}x_{\bar t}^T P_{\bar t} x_{\bar t} + \frac{1}{2}\sum_{t=\bar t}^{N-1} \ub{t}^T G_{t+1} \ub{t}, \label{eq:lem:F_G_H_cost:cost}
\end{equation}
where $P_{\bar t}$, $G_{t+1}$ and $K_{t+1}$ are computed in the Riccati factorization in Algorithm~\ref{alg:factorization}.
\end{lemma}
\begin{proof}
\label{proof:lem:F_G_H_cost}
For the proof of Lemma~\ref{lem:F_G_H_cost}, see Appendix~\ref{app:proof:lem:F_G_H_cost}.
\end{proof}


\begin{theorem}
\label{thm:reduce_subprob}
An MPC problem given on the form~\eqref{eq:sub_problem_dscr} with unknown $P_{N}$ can be reduced into a smaller MPC problem in $\xh{} \in \mathbb{R}^{\nx}$ and $\uh{} \in \mathbb{R}^{\nr}$, with $\nr \leq \nx$ using the Riccati factorization.
%
%
The reduced problem has the cost function
\begin{equation}
\Vh{}\left( \xh{}, \uh{} \right) = \frac{1}{2}\xh{}^T \Qh{x} \xh{} + \frac{1}{2}\uh{}^T \Qh{\uu} \uh{},
\end{equation}
and the dynamics from the initial state to the final state in the batch are given by
\begin{align}
x_{N} &= \Ah{}\xh{} + \Bh{} \uh{},
\end{align}
where $\Qh{x}$, $\Qh{\uu}$, $\Ah{}$ and $\Bh{}$ are given by~\eqref{eq:Qxh_definition},~\eqref{eq:Ah_def} and~\eqref{eq:Bh_def}.

\end{theorem}
\begin{proof}
Let the MPC problem given on the form~\eqref{eq:sub_problem_dscr} be factored for $P_{N}=0$ using the Riccati factorization given by Algorithm~\ref{alg:factorization}. This gives the feedback $\uu_{0,t} = K_{0,t+1}x_t$ for $t \in \intset{0}{N-1}$ which is optimal if $P_N=0$. It will now be investigated how $u_t$ is affected when $P_N \neq 0$. Let the contribution from the unknown $P_N$ be denoted $\ub{t} \in \mathbb{R}^{\nuu}$, giving the control signal
\begin{equation}
\uu_{t} = K_{0,t+1}x_{t}+\ub{t}, \; t \in \intset{0}{N-1}. \label{eq:control_signal_def}
\end{equation}
Note that $\ub{t}$ is a full $\nuu$ vector, hence there is no loss of generality in this assumption. Using~\eqref{eq:control_signal_def}, the states $\timestack{x}$ along the horizon can be expressed as
\begin{equation}
\timestack{x} = \timestack{A}x_0 + \timestack{B} \timestack{\ub{}},
\end{equation}
with cost function for $P_N = 0$ given by Lemma~\ref{lem:F_G_H_cost}, \ie,
\begin{equation}
\Vb{}\left(x_0,\timestack{u} \right) = \frac{1}{2}x_0^T P_{0} x_0 + \frac{1}{2}\timestack{\ub{}}^T\timestack{\Qb{\ub{}}}\timestack{\ub{}}. \label{eq:proof:thm:reduce_subprob:cost}
\end{equation}
Here $\timestack{\Qb{\ub{}}}$, $\timestack{A}$ and $\timestack{B}$ are given by
\begin{align}
\small
\timestack{\Qb{\ub{}}}&\triangleq \begin{bmatrix}
G_{0,1} \\
 & \ddots \\
 & & G_{0,N}
\end{bmatrix}, \;
\timestack{A} \triangleq \begin{bmatrix}
I \\ A_0 + B_0 K_{0,1} \\ \vdots \\ \prod_{t=0}^N \left(A_t + B_t K_{0,t+1} \right)
\end{bmatrix},  \notag \\
\timestack{B} &\triangleq \begin{bmatrix}
0 & 0 & \hdots & 0\\
B_0 \\
(A_1 + B_1 K_{0,2})B_0 & B_1 & & \vdots \\
\vdots & & \ddots \\
\prod_{t=1}^{N-1} \left(A_t + B_t K_{0,t+1} \right)B_0 & & \hdots & B_{N-1}
\end{bmatrix}.\label{eq:proof:thm:reduce_subprob:QuA}
\end{align}
Now, let $\Ah{}$ and $\Sc$ be the last block rows in $\timestack{A}$ and $\timestack{B}$, respectively. The dynamics from $x_0$ to $x_N$ is then given by
\begin{equation}
x_N = \Ah{} x_0 + \Sc \timestack{\ub{}}, \label{eq:proof:thm:reduce_subprob:batch_dynamics}
\end{equation}
which together with~\eqref{eq:proof:thm:reduce_subprob:cost}, and the fact that $P_N$ is the cost for the final state $x_N$, constitutes a new optimization problem
\begin{equation}
\minimize{\frac{1}{2}x_0^TP_0x_0+\frac{1}{2}\timestack{\ub{}}^T\timestack{\Qb{\ub{}}}\timestack{\ub{}}+\frac{1}{2}x_N^TP_Nx_N}{x_0,\timestack{\ub{}},x_N}{x_0 &= \xh{} \\ x_N &= \Ah{}x_0+\Sc \timestack{\ub{}}.} \label{eq:proof:thm:reduce_subprob:dense_prob}
\end{equation}
This is an MPC problem with prediction horizon $1$ (one step from the initial state to the final state), and can be solved using the Riccati recursion, giving
\begin{align}
\Fb{} &= P_0 + \Ah{}^T P_N \Ah{} \label{eq:proof:thm:reduce_subprob:Fbar}\\
\Gb{} &= \timestack{\Qb{\ub{}}} + \Sc^T P_N \Sc \\
\Hb{} &= \Ah{}^T P_N \Sc \\
\Gb{}\Kb{} &= - \Hb{}^T, \label{eq:proof:thm:reduce_subprob:GK}
\end{align}
where~\eqref{eq:proof:thm:reduce_subprob:GK} can be written
\begin{equation}
\left(\Qb{\ub{}}+ \Sc^T P_N \Sc \right)\Kb{} = -\Sc^T P_N \Ah{}. \label{eq:proof:thm:reduce_subprob:GK_written_out}
\end{equation}

Let $U_1$ be an orthonormal basis of $\range{\Sc^T}$ and let $U_2$ be an orthonormal basis of $\range{\Sc^T}^\perp$ given by the singular value decomposition of $\Sc^T P_N \Sc$, \ie,
\begin{equation}
\Sc^T P_N \Sc = \begin{bmatrix}
U_1 & U_2
\end{bmatrix}\begin{bmatrix}
\Sigma(P_N) & 0 \\ 0 & 0
\end{bmatrix} \begin{bmatrix}
U_1^T \\ U_2 ^T
\end{bmatrix}. \label{eq:ST_svd}
\end{equation}
Then $U = [U_1,\; U_2]$ is an orthonormal basis for $\mathbb{R}^{N \nuu}$, and by using the identity $UU^T = I$~\eqref{eq:proof:thm:reduce_subprob:GK_written_out} can equivalently be written
\begin{align}
&\left( U U^T \Qb{\ub{}} + U_1 \Sigma(P_N) U_1^T \right)\Kb{} =  \begin{bmatrix}
U_1 & U_2
\end{bmatrix} \begin{bmatrix} \Gamma (P_N) \\ 0
\end{bmatrix} \notag \iff \\
&\begin{cases}
\left(U_1^T \timestack{{\bar Q_{\ub{}}}} + \Sigma(P_N) U_1^T \right) \Kb{} = \Gamma(P_N) \\
U_2^T \timestack{{\bar Q_{\ub{}}}} \Kb{} = 0 \iff \Kb{} = \timestack{\inv{\bar Q_{\ub{}}}} U_1 Z, \; Z \in \mathbb{R}^{\nr \times \nr}
\end{cases} \label{eq:proof:thm:reduce_subprob:u1u1T}
\end{align}
where $\nr = \textrm{dim}\range{\Sc^T} \leq \nx$. Here $U_1^TU_1 = I$, $U_1^TU_2 = 0$ and $U_2^T U_2 = I$ was used to reduce the size of the system of equations. Inserting $\Kb{} = \timestack{\inv{\bar Q_{\ub{}}}} U_1 Z$ into~\eqref{eq:proof:thm:reduce_subprob:u1u1T} gives
\begin{align}
&\left(U_1^T \timestack{{\bar Q_{\ub{}}}} +  \Sigma(P_N) U_1^T \right)\timestack{\inv{\bar Q_{\ub{}}}} U_1 Z  = \Gamma(P_N) \\
&\iff \left(I+ \Sigma(P_N) U_1^T \timestack{\inv{\bar Q_{\ub{}}}} U_1 \right) Z = \Gamma(P_N). \label{eq:proof:thm:reduce_subprob:U1Z}
\end{align}
Now multiply \eqref{eq:proof:thm:reduce_subprob:U1Z} with $U_1^T \timestack{\inv{\bar Q_{\ub{}}}}  U_1$ from the left, giving
\begin{equation}
\begin{split}
\left(U_1^T \timestack{\inv{\bar Q_{\ub{}}}}  U_1 + U_1^T \timestack{\inv{\bar Q_{\ub{}}}}  U_1 \Sigma(P_N) U_1^T \timestack{\inv{\bar Q_{\ub{}}}}  U_1 \right) Z = \\U_1^T \timestack{\inv{\bar Q_{\ub{}}}}  U_1 \Gamma(P_N). \label{eq:proof:thm:reduce_subprob:reducedGK}
\end{split}
\end{equation}
Next, choose $T \in \mathbb{R}^{\nr \times \nx}$ with full rank such that $U_1T = \Sc^T$, and let $Z = T \Kh{}$ for some $\Kh{} \in \mathbb{R}^{\nx \times \nx}$. With this choice of $Z$ inserted in~\eqref{eq:proof:thm:reduce_subprob:reducedGK} and multiplying from the left with $T^T$ gives
\begin{equation}
\begin{split}
\left(\Sc \timestack{\inv{\bar Q_{\ub{}}}}  \Sc^T + \Sc \timestack{\inv{\bar Q_{\ub{}}}}  U_1 \Sigma(P_N) U_1^T \timestack{\inv{\bar Q_{\ub{}}}} \Sc^T \right) \Kh{} = \\ \Sc \timestack{\inv{\bar Q_{\ub{}}}}  U_1 \Gamma(P_N). \label{eq:proof:thm:reduce_subprob:SQS}
\end{split}
\end{equation}
Using $\Gamma(P_N) = -T P_N \Ah{}$, the right hand side of~\eqref{eq:proof:thm:reduce_subprob:SQS} can be re-written as 
\begin{equation}
\begin{split}
\Sc \timestack{\inv{\bar Q_{\ub{}}}}  U_1 \Gamma(P_N) = -\Sc \timestack{\inv{\bar Q_{\ub{}}}}  U_1 T P_N \Ah{} = -\Sc \timestack{\inv{\bar Q_{\ub{}}}} \Sc^T P_N \Ah{},
\end{split}
\end{equation}
and by using this expression together with $\Sc^T P_N \Sc = U_1 \Sigma(P_N) U_1^T$,~\eqref{eq:proof:thm:reduce_subprob:SQS} can be written
\begin{equation}
\begin{split}
\left(\Sc \timestack{\inv{\bar Q_{\ub{}}}}  \Sc^T + \Sc \timestack{\inv{\bar Q_{\ub{}}}}  \Sc^T P_N  \Sc \timestack{\inv{\bar Q_{\ub{}}}} \Sc^T \right) \Kh{} = \\ -\Sc \timestack{\inv{\bar Q_{\ub{}}}} \Sc^T  P_N \Ah{}. \label{eq:proof:thm:reduce_subprob:SQSdone}
\end{split}
\end{equation}
By introducing the variables
\begin{align}
\Qh{u} &\triangleq \Sc \timestack{\inv{\bar Q_{\ub{}}}}  \Sc^T  , \quad
\Bh{} \triangleq \Sc \timestack{\inv{\bar Q_{\ub{}}}}  \Sc^T \label{eq:Bh_def} \\
\Gh{} & \triangleq \Qh{u} + \Bh{}^T P_N \Bh{}, \quad
\Hh{} \triangleq \Ah{}^T P_N \Bh{},
\end{align}
the equation in~\eqref{eq:proof:thm:reduce_subprob:SQSdone} can be written as
\begin{equation}
\Gh{} \Kh{} = - \Hh{}^T. \label{eq:proof:thm:reduce_subprob:GhKh}
\end{equation}
Hence, by also defining 
\begin{align}
\Qh{x}  &\triangleq P_0, \quad \Fh{} \triangleq \Fb{}, \label{eq:Qxh_definition}  \\
\Ah{} &\triangleq \prod_{t=0}^{N-1}\left( A_{t}+B_{t}K_{0,t+1} \right) \label{eq:Ah_def}
\end{align}
the equations~\eqref{eq:proof:thm:reduce_subprob:Fbar} to~\eqref{eq:proof:thm:reduce_subprob:GK} can be written as
\begin{align}
\Fh{} &= \Qh{x} + \Ah{}^T P_N \Ah{}\\
\Gh{} &= \Qh{u} + \Bh{}^T P_N \Bh{} \\
\Hh{} &= \Ah{}^T P_N \Bh{} \\
\Gh{} \Kh{} &= -\Hh{}^T
\end{align}
which can be identified as the KKT condition for an MPC problem on the form~\eqref{eq:proof:thm:reduce_subprob:dense_prob}, but with smaller control signal dimension ${\nr \leq \nx}$.

\end{proof}
\begin{remark}
The preliminary $P_N$ can be chosen as any $P_N \succeq 0$, \eg, the infinite horizon LQ-cost. For presentation reasons, the choice $P_N = 0$ is made in the proof of Theorem~\ref{thm:reduce_subprob}.  
\end{remark}
\begin{remark}
If $\Sc^T$ is rank deficient then $U_1\in \mathbb{R}^{N\nuu \times \nr}$ will have $\nr < \nx$ columns. Hence $\Gh{}$ is singular and $\Kh{}$ non-unique in~\eqref{eq:proof:thm:reduce_subprob:GhKh}. Problems of this form has been studied in, \eg,~\cite{axehill08:thesis}.
\end{remark}
\begin{remark}
Even though the problem~\eqref{eq:proof:thm:reduce_subprob:dense_prob} is used in the proof of Theorem~\ref{thm:reduce_subprob}, it is not explicitly used in Algorithm~\ref{alg:prel_fact}.
\end{remark}

The formal validy of the reduction of each subproblem $i \in \intset{0}{p}$ is ensured by Theorem~\ref{thm:reduce_subprob}, while the computational procedure is summarized in Algorithm~\ref{alg:prel_fact}, which is basically a Riccati factorization. Note that the final subproblem $\p$ can be factored exactly directly, since~$P_{N_\p,\p}=Q_{x,N}$ is known. Hence, in that subproblem there is no~$\uh{\p}$ since the subproblem is only dependent on the initial value~$\xh{\p}$.
\begin{algorithm}[h!]
  \caption{Reduction using Riccati factorization} \label{alg:prel_fact}
  \begin{algorithmic}[1]
    \STATE $P_N := 0$
    \STATE $\hat Q_u := 0$
    \STATE $\Vk{N} := I$
    \FOR{$t=N-1,\ldots,0$}
    \STATE $F_{t+1} := Q_{x,t} + A^T_tP_{t+1}A_t$\label{alg:prel_fact:line:F_comp} \\
    \STATE $G_{t+1} := Q_{u,t} + B^T_tP_{t+1}B_t$ \\
    \STATE $H_{t+1} := Q_{x\uu,t} + A^T_tP_{t+1}B_t$\label{alg:prel_fact:line:H_comp} \\
    \STATE Compute and store a factorization of $G_{t+1}$.\label{alg:prel_fact:line:factorize_G}
    \STATE Compute a solution $K_{t+1}$ to \\ 
    $G_{t+1}K_{t+1} = -H^T_{t+1}$ \\
    \STATE Compute a solution $\Mm{t+1}$ to \\ 
    $G_{t+1}\Mm{t+1} = -B_t^T \Vk{t+1}$ \\
    \STATE $\Vk{t} := \left( A_t^T + K_{t+1}^T B_t^T \right)\Vk{t+1}$
    \STATE $P_{t} := F_{t+1} - K^T_{t+1}G_{t+1}K_{t+1}$\label{alg:prel_fact:line:P_comp}
    \STATE $\hat Q_u := \hat Q_u + \Mm{t+1}^T G_{t+1} \Mm{t+1}$
    \ENDFOR
    \STATE $\Ah{} := \Vk{0}^T$
    \STATE $\Bh{} := \Qh{u}$
  \end{algorithmic}
\end{algorithm}
%


\subsection{Constructing the reduced MPC problem}
\label{subsec:red_mpc_prob}
All subproblems $i \in \intset{0}{p}$ can be condensed to depend only on the initial value $\xh{i}$ and $\uh{i}$ according to Theorem~\ref{thm:reduce_subprob} and Section~\ref{subsec:red_sub_prob}. The variable $\uh{i}$ represents the unknown part of the control signals $u_{t,i}$ that are due to the initially unknown $P_{N_i,i}$ and can be interpreted as a new control signal for batch~$i$. The condensed subproblems can be combined into an optimization problem equivalent to the original equality constrained MPC problem~\eqref{eq:org_eqc_problem}, but with prediction horizon~$\p < N$ and control signal dimension $\nr$, \ie,
\begin{equation}
\minimize{\frac{1}{2}\sum_{i=0}^{\p-1}\begin{bmatrix}
    \xh{i}^T & \uh{i}^T
    \end{bmatrix} \Qh{i}\begin{bmatrix}
    \xh{i} \\  \uh{i}
\end{bmatrix} + \frac{1}{2} \xh{\p}^T \Qxh{\p} \xh{\p}}{\timestack{\xh},\timestack{\uh}}{\xh{0}&=\bar x_0 \\ \xh{i+1} &= \Ah{i} \xh{i} + \Bh{i} \uh{i}, \; i \in \intset{0}{\p-1}.}
\label{eq:red_mpc_problem}
\end{equation}
This problem is on the same form as~\eqref{eq:org_eqc_problem} but the number of unknowns are reduced. The dynamics equations $\xh{i+1}=\Ah{i}\xh{i} + \Bh{i} \uh{i}$ are due to the fact that $x_{N_i,i} = x_{0,i+1}$ per definition in the splitting of the time horizon in Section~\ref{subsec:sub_prob}.  Hence, an MPC problem $\MPC{N}$ of prediction horizon length $N$ can be reduced, by using Riccati factorizations in each subproblem, to an MPC problem $\MPC{\p}$ on the same form but with shorter prediction horizon and lower control signal dimension. Fig.~\ref{fig:mpc_red_struct} illustrates this reduction procedure, where the notation $\MPCsub{i}{}$ is introduced to denote subproblem $i$ given by~\eqref{eq:sub_problem_dscr}, with prediction horizon $N_i$.

To solve the original problem, \ie solving all subproblems, the reduced problem $\MPC{\p}$ is solved using the Riccati recursion to obtain the optimal $\Ph{i}$ and $\xh{i}$ for $i \in \intset{0}{\p}$. Thereafter the subproblems are solved with $P_{N_i,i} = \Ph{i+1}$ for $i \in \intset{0}{\p-1}$ and $x_{0,i} = \xh{i}$ for $i\in \intset{0}{\p}$ using Algorithms~\ref{alg:factorization} to~\ref{alg:fwd_rec_dual}. 
\begin{figure}
\centering
\def\svgwidth{0.9\columnwidth}
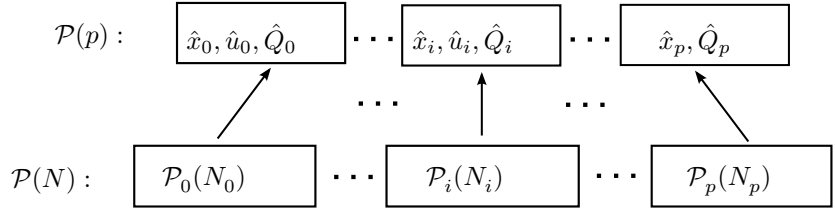
\caption{The original MPC problem $\MPC{N}$ can be reduced to a smaller problem $\MPC{\p}$ with the same structure but with shorter prediction horizon.}
\label{fig:mpc_red_struct}
\end{figure}

\section{Parallel Riccati Recursion}
\label{sec:par_ric}
The reduced problem~\eqref{eq:red_mpc_problem} can be reduced repeatedly using the theory presented in Section~\ref{sec:prob_red} until a smaller MPC problem with desired length of the prediction horizon is obtained. This structure is similar to what was made in~\cite{nielsen14_parallel_mpc_arxive}, but one of the differences here is the way it is performed using Riccati recursions which allows for a complete exploitation of structure. This procedure is depicted in Fig.~\ref{fig:arb_tree_struct}, where $\MPCsub{i}{k}$ denotes subproblem $i$ on level $k$ in the reduction tree. Hence, the tree structure is built in parallel.

When the top problem $\MPC{\p_{m-1}}$ is solved, the solution can be propagated to its children $\MPCsub{i}{m-1}$ for ${i \in \intset{0}{\p_{m-1}}}$. Each subproblem receives $\Ph{i+1}$ and $\xh{i}$ from its parent and as soon as these are known to the subproblem, it can be solved independently from the other subproblems at the same level of the tree. This procedure consists of two steps; reducing the original problem $\MPC{N}$ in parallel to $\MPC{\p_{m-1}}$, and then propagating the solution of $\MPC{\p_{m-1}}$ down in the tree. These main steps are summarized in Algorithm~\ref{alg:build_tree} and~\ref{alg:propagate_solution}. Since all levels can be solved in parallel using the Riccati recursion, and the result at the bottom level is identical to if a serial Riccati recursion was used to solve $\MPC{N}$, the Riccati recursion has been parallelized.
\begin{figure}[htb]
\centering
\def\svgwidth{\columnwidth}
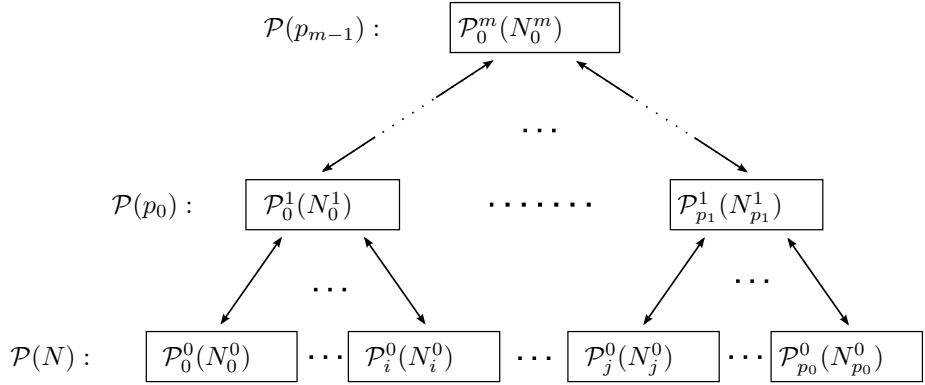
\caption{The original MPC problem $\MPC{N}$ can be reduced repeatedly in several steps using Riccati factorizations on the way up in the tree. When the top problem has been reached and solved, the solution can be propagated back down in the tree until the bottom level is solved.}
\label{fig:arb_tree_struct}
\end{figure}


\subsection{Algorithms for parallel Riccati recursion}
\label{subsec:algs}
In this section algorithms for computing the Riccati recursion in parallel are presented. Beyond what is presented here, as observed already in~\cite{axehill:_towar_mpc}, standard parallel linear algebra can be used in many computations in the serial Riccati recursion to boost performance even further. This has however not been utilized in this work. 

In Algorithm~\ref{alg:build_tree}, the original problem is reduced in parallel in several steps to an MPC problem with prediction horizon~$\p_{\min}$. Assume, for simplicity, that all subproblems are of equal length $N_s$ and that $N=N_s^m$ for some $1 < m \in \intnums$. Then this reduction can be made in $m$ steps, provided that $ N/N_s=m-1$ processing units are available. Hence, the reduction algorithm has $\Ordo{\log N}$ complexity growth. 
\begin{algorithm}[h!]
  \caption{Parallel reduction of MPC problem} \label{alg:build_tree}
  \begin{algorithmic}[1]
  	\STATE Initiate level counter $k:=0$
  	\STATE Initiate the first number of subsystems $p_{-1} = N$
  	\STATE Set the minimal number of sub problems $p_{\min}$
    \WHILE{$p_k > p_{\min}$}
    	\STATE Compute desired $p_k$ to define the number of sub problems (with $p_k < p_{k-1}$)
    	\STATE Split the prediction horizon $0,\ldots, p_{k-1}$ in $p_k+1$ segments
        $0,\ldots, N_0^k$ up to $0,\ldots, N_{p_k}^k$
        \STATE Create sub problems $i=0,\ldots,p_k$ for each time batch according to Section~\ref{subsec:sub_prob}
   	    \STATE \textbf{parfor} $i = 0,\ldots,p_k$ \textbf{do} \label{alg:line:build_tree_loop}
        	\STATE \hspace{2ex} Reduce subproblem $i$ according to Algorithm~\ref{alg:prel_fact}
      		\STATE \hspace{2ex} Propagate $\Ah{i}$, $\Bh{i}$, $\Qxh{i}$ and $\Quh{i}$ to next level
     	\STATE \textbf{end parfor} \label{alg:line:build_tree_loop_end}
     	\STATE Update level counter $k:= k + 1$
    \ENDWHILE
    \STATE Compute maximum level number $k:= k-1$ 
  \end{algorithmic}
\end{algorithm}

In Algorithm~\ref{alg:propagate_solution} the solution (\ie $\xh{i}^m$ and $\hat P_{i}^m$ for $i \in \intset{0}{\p_m-1}$) to the problem $\MPC{\p_{m-1}}$ in the tree structure in Fig.~\ref{fig:arb_tree_struct} is propagated down in the tree to the leaves $\MPCsub{i}{0}$, $i \in \intset{0}{\p_0}$. All subproblems can be solved using only information from their parents, and hence each level in the tree can be solved completely in parallel. The propagation of the solution from the top level to the bottom level can thus be made in $m$ steps provided that $m-1$ processing units are available. Since both Algorithm~\ref{alg:build_tree} and~\ref{alg:propagate_solution} are solved in $\Ordo{\log N}$ complexity, the solution to the equality constrained MPC problem~\eqref{eq:org_eqc_problem} can be computed in $\Ordo{\log N}$ complexity growth. The solution to the original inequality constrained problem~\eqref{eq:min_problem} is obtained by solving a sequence of problems of the form in~\eqref{eq:org_eqc_problem}. Since the length of this sequence is independent of whether~\eqref{eq:org_eqc_problem} is solved serially or in parallel, the performance gain obtained by this work is directly transferred to the overall solution time.
\begin{algorithm}
  \caption{Parallel propagation of solution} \label{alg:propagate_solution}
  \begin{algorithmic}[1]
  	\STATE Initialize the first parameter as $\bar x_0$
  	\STATE Get level counter $k$ from Algorithm~\ref{alg:build_tree}
    \WHILE{$k \geq 0$}
	    \STATE \textbf{parfor} {$i=0,\ldots,p_k$} \textbf{do}\label{alg:line:propagate_solution_loop}
    		\STATE \hspace{2ex} Compute the full factorization and the primal \\ \hspace{3ex}variables according to Algorithm~\ref{alg:factorization} and~\ref{alg:fwd_rec} \label{alg:propagate_solution:line:compute_factorization}
    	\STATE \textbf{end parfor} \label{alg:line:propagate_solution_loop_end}
    	\IF{k==0}
	    \STATE \textbf{parfor} {$i=0,\ldots,p_0$} \textbf{do}\label{alg:line:propagate_solution_dual_loop}
    		\STATE \hspace{2ex} Compute the dual variables corresponding to \\ \hspace{2ex} equality and inequality constraints using \\ \hspace{2ex} Algorithm~\ref{alg:fwd_rec} and~\ref{alg:fwd_rec_dual}
    	\STATE \textbf{end parfor} \label{alg:line:propagate_solution_loop_dual_end}
    	\ENDIF
   		\STATE Update level counter $k:=k-1$
    \ENDWHILE
  \end{algorithmic}
\end{algorithm}


\subsection{Numerical results}
\label{subsec:num_res}
The algorithms presented in Section~\ref{subsec:algs} have been implemented in \textsc{Matlab}. The algorithms are implemented serially and run using only one computational thread, but the information flow is done in the same way as for a fully parallel implementation. The computation time for a truly parallel implementation has been computed by summing over the maximum computation time for each level in the tree. This estimate does not take the communication latencies into account, but these are assumed to be negligible in comparison to the actual computations. The performance of the parallel Riccati algorithm in this work is compared with the serial Riccati recursion, which is considered a state-of-the-art serial method.

The computation times when computing the Newton step for random MPC problems for stable LTI systems of order $\nx = 7$, $\nuu = 5$ and using $N_s = 2$ are presented in Fig.~\ref{fig:num_res}. The computation times are averaged over $10$ random systems of the same order. The dash-dotted line is the computation times for the serial Riccati recursion and the solid line is the new parallel Riccati recursion algorithm. The result is plotted in a log-log scale to compare the complexity growth. For prediction horizons larger than ${N  \approx 16}$ the parallel Riccati recursion outperforms the serial one.
\begin{figure}[htb!]
\centering
\includegraphics[width=0.85\columnwidth]{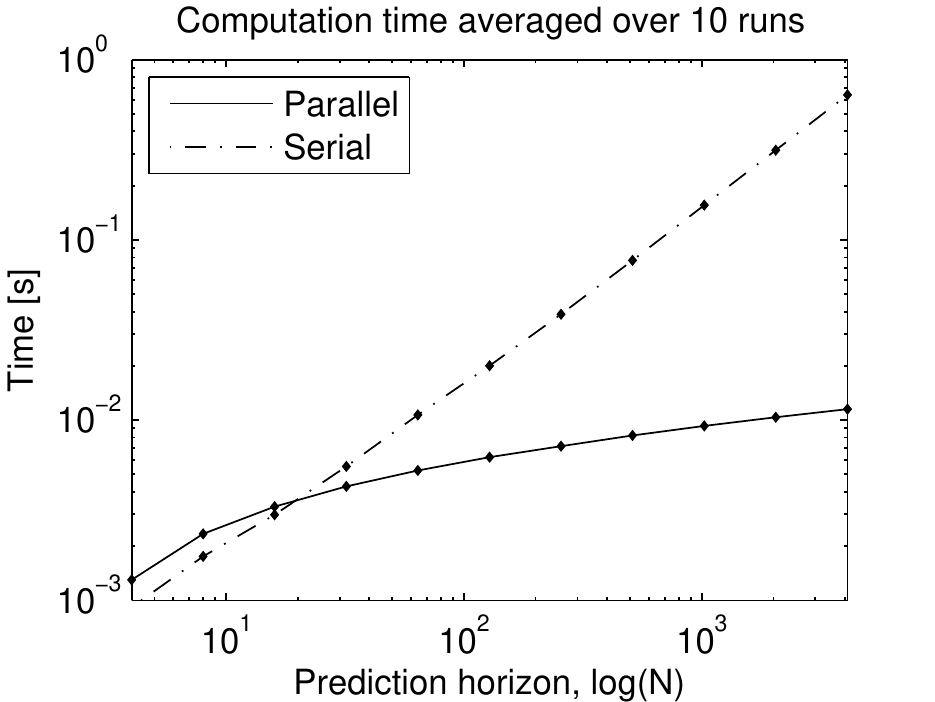}
\caption{Computation times for the parallel Riccati recursion (solid) and for the serial (dash-dotted) when computing Newton steps for random MPC problems with $\nx=7$, $\nuu = 5$ and $N_s = 2$. The parallel Riccati recursion outperforms the serial one for $N \gtrsim 16$.}
\label{fig:num_res}
\end{figure}
In Fig.~\ref{fig:num_res_Ns_3} the computation times for systems of the same order as in Fig.~\ref{fig:num_res} but with $N_s=3$ has been plotted. Here the parallel method outperforms the serial one for $N \geq 9$. How to choose the length $N_s$ of the batches to obtain the lowest possible computation time is not investigated here. However, similar to what is described in~\cite{axehill13_controlling_sparisty_arxive}, the optimal choice depends on, \eg, the problem and the hardware which the algorithm is implemented on.
\begin{figure}[htb!]
\centering
\includegraphics[width=0.85\columnwidth]{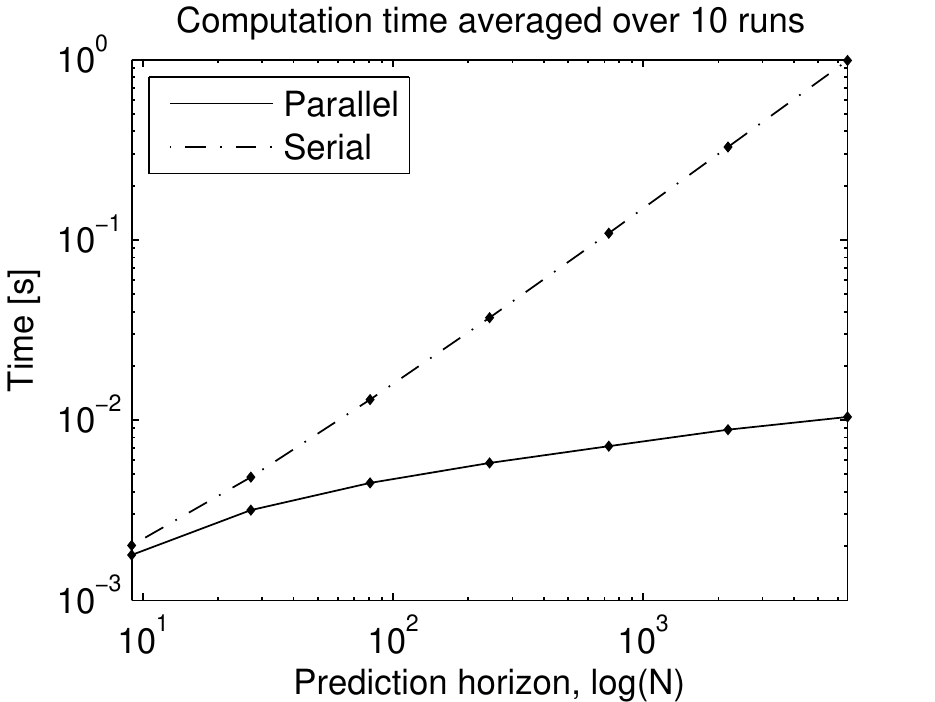}
\caption{Computation times for the parallel Riccati recursion (solid) and for the serial (dash-dotted) when computing Newton steps for random MPC problems with $\nx=7$, $\nuu = 5$ and $N_s = 3$. The parallel Riccati recursion outperforms the serial one for $N \geq 9$.}
\label{fig:num_res_Ns_3}
\end{figure}

The simulations were performed on an 
Intel Xeon CPU X5675 @ 3.07 GHz running Linux (version 2.6.32-431.5.1.el6.x86\_64) and \textsc{Matlab} (8.0.0.783 (R2012b)).

\section{Conclusions}
\label{sec:conclusion}
This work introduces theory and algorithms for parallelization of the Riccati recursion. It is shown that the Newton step corresponding to an equality constrained MPC problem can be solved directly (non-iteratively) in parallel using Riccati recursions that fully exploit the structure from the MPC problem.  The algorithms have been implemented in \textsc{Matlab} and have been used to compute the Newton step for random MPC problems with stable LTI systems as a proof of concept that the theory works in practice, and to compare performance with a serial state-of-the-art Riccati algorithm. The resulting parallel algorithm has a complexity growth as low as $\Ordo{\log N}$, where $N$ is the length of the prediction horizon. For future work the structure in the updates of the feedback gain $K_t$ will be investigated further to improve performance even more. 


\appendix
\subsection{Proof of Lemma~\ref{lem:F_G_H_cost}}
\label{app:proof:lem:F_G_H_cost}
Assume that~\eqref{eq:lem:F_G_H_cost:cost} holds for an arbitrary $\bar t + 1 \in \intset{1}{N-1}$. Then, the cost at $t= \thh$ is given by
\begin{equation}
\frac{1}{2}\begin{bmatrix}
x_{\thh}^T & \uu_{\thh}^T
\end{bmatrix} \begin{bmatrix}
Q_{x,\thh} & Q_{x\uu,\thh} \\
Q_{x\uu,\thh}^T & Q_{\uu,\thh}
\end{bmatrix}\begin{bmatrix}
x_{\thh} \\ \uu_{\thh}
\end{bmatrix}+\Vb{}\left(x_{\thh+1},\timestack{\bar \uu}\right). \label{eq:proof:lem:F_G_H_cost:cost}
\end{equation}
By inserting $x_{\thh+1}=A_{\thh}x_{\thh}+B_{\thh}\uu_{\thh}$ into~\eqref{eq:proof:lem:F_G_H_cost:cost}, the cost can be written
\begin{equation}
\frac{1}{2}\begin{bmatrix}
x_{\thh}^T & \uu_{\thh}^T
\end{bmatrix} \begin{bmatrix}
F_{\thh+1} & H_{\thh+1} \\
H_{\thh+1}^T & G_{\thh+1}
\end{bmatrix}\begin{bmatrix}
x_{\thh} \\ \uu_{\thh}
\end{bmatrix}+\frac{1}{2}\sum_{t=\thh+1}^{N-1} \ub{t}^T G_{t+1} \ub{t}, \label{eq:proof:lem:costfcn}
\end{equation}
where $F_{\thh+1}$, $H_{\thh+1}$ and $G_{\thh+1}$ are given by the Riccati recursion. Finally, using the control law $u_{\thh} = K_{\thh+1}x_{\thh}+\ub{\thh}$ and the definition of $F_{\thh+1}$, $H_{\thh+1}$ and $G_{\thh+1}$ gives the cost function
\begin{equation}
\Vb{}\left(x_{\thh},\timestack{\ub{}}\right) = \frac{1}{2}x_{\thh}^T P_{\thh} x_{\thh} + \frac{1}{2}\sum_{t=\thh}^{N-1} \ub{t}^T G_{t+1} \ub{t}.
\end{equation}
Note that the cross terms between $x_{\thh}$ and $\ub{\thh}$ in the cost function~\eqref{eq:proof:lem:costfcn} vanishes since $G_{\thh+1}K_{\thh+1}=-H_{\thh+1}^T$.
Equation~\eqref{eq:lem:F_G_H_cost:cost} holds specifically for $t=N-1$ when $P_N = 0$, and hence Lemma~\ref{lem:F_G_H_cost} follows by mathematical induction.

\bibliography{axe_abrv,IEEEabrv,ianFull}

\end{document}